\documentclass[sn-mathphys]{sn-jnl}

\jyear{2021}%

\theoremstyle{thmstyleone}%
\newtheorem{theorem}{Theorem}
\newtheorem{proposition}[theorem]{Proposition}%
\newtheorem{corollary}[theorem]{Corollary}%

\theoremstyle{thmstyletwo}%
\newtheorem{remark}{Remark}%

\theoremstyle{thmstylethree}%

\usepackage{empheq}

\usepackage{babel}
\usepackage{paralist}

\newcommand{\vx}{\boldsymbol x}
\newcommand{\vv}{\boldsymbol v}
\newcommand{\vu}{\boldsymbol u}
\newcommand{\vc}{\boldsymbol c}

\newcommand{\tp}{\boldsymbol p}

\newcommand{\intR}{\int_{\mathbb R^3}}
\newcommand{\intS}{\int_{\mathbb S^2}}
\newcommand{\dd}{\,\mathrm{d}}

\begin{document}

\title[Higher-order Maxwell--Stefan model of diffusion]{Higher-order Maxwell--Stefan model of diffusion}

\author*[1]{\fnm{B\'{e}r\'{e}nice} \sur{Grec}}\email{berenice.grec@u-paris.fr}

\author[2]{\fnm{Srboljub} \sur{Simi\'{c}}}\email{ssimic@uns.ac.rs}
\equalcont{These authors contributed equally to this work.}

\affil*[1]{\orgdiv{Universit\'e Paris Cit\'e, CNRS}, \orgname{MAP5}, \orgaddress{
 \city{Paris}, \postcode{F-75006}, \country{France}}}

\affil[2]{\orgdiv{University of Novi Sad}, \orgname{Faculty of Sciences, Department of Mathematics and Informatics}, \orgaddress{\street{Trg Dositeja Obradovi\'{c}a 4}, \city{Novi Sad}, \postcode{21000}, \country{Serbia}}}

\abstract{\begin{hyphenrules}{nohyphenation}The paper studies a higher-order diffusion model of Maxwell--Stefan kind. The model is based upon higher-order moment equations of kinetic theory of mixtures, which include viscous dissipation in the model. Governing equations are analyzed in a scaled form, which introduces the proper orders of magnitude of each term. In the so-called diffusive scaling, the Mach and Knudsen numbers are assumed to be of the same small order of magnitude. In the asymptotic limit when the small parameter vanishes, the model exhibits a coupling between the species' partial pressure gradients, which generalizes the classical model. Scaled equations also lead to a higher-order model of diffusion with correction terms in the small parameter. In that case, the viscous tensor is determined by genuine balance laws.\end{hyphenrules}}


\maketitle

\vspace{-0.8cm}
\bmhead{Conflict of interest}
{\footnotesize On behalf of all authors, the corresponding author states that there is no conflict of interest.}

\bmhead{Acknowledgments}
{\footnotesize This paper was prepared during the stay of Srboljub Simi\'{c} at Universit\'{e} Paris Cit\'{e} thanks to ``Guest researchers' faculty programme 2022''. The research was also financially supported 
 by the Ministry of Science, Technological Development and Innovation of the Republic of Serbia (Grant No. 451-03-47/2023-01/200125) (S.S.).}

\newpage

\section{Introduction}

Macroscopic description of diffusion usually relies on two well-known models, the Fick and Maxwell--Stefan ones. They are both based upon the mass conservation of species, but differ in the way diffusion mechanism is described, and thus the system closed. In the Fick model, the diffusion fluxes are determined by constitutive relations of phenomenological nature -- they are linearly related to the gradients of species' chemical potentials/densities \cite{giovangigli1999multicomponent}. In the Maxwell--Stefan model, diffusion is regarded as a source of momentum exchange between the species, which is balanced by the gradients of partial pressures \cite{krishna1997maxwell,bothe2011maxwell}. The system is thus closed through a kind of momentum balance equations for the species. While the Fick model fits well into linear thermodynamics of irreversible processes \cite{de1962non}, the Maxwell--Stefan model fits in the framework of classical cross-diffusion models \cite{hutridurga2018existence,daus2020exponential}. 

It is remarkable that both diffusion models mentioned above may be derived starting from a more sophisticated model---the Boltzmann equations of the kinetic theory of mixtures. To derive the Fick model, one has to apply the Chapman--Enskog method, i.e perform the asymptotic expansion of the velocity distribution function and the conservation laws derived from the Boltzmann equations, using the Knudsen number as small parameter \cite{Chapman-Cow,FGolse,briant2020rigorous}. This amounts to a hydrodynamic limit of kinetic equations and inherits a physical assumption that the process occurs in the neighborhood of local equilibrium state. 

On the other hand, to derive the Maxwell--Stefan model, one has to extend the system of moment equations and include the momentum balances for the species. The Maxwell--Stefan relations are then obtained as an asymptotic limit of the momentum balance laws in so-called diffusive scaling \cite{boudin2015maxwell,boudin2017maxwell}. The diffusive scaling reflects the assumption that macroscopic velocities in diffusion processes are small compared to molecular velocities, while asymptotic limit amounts to neglecting inertia terms and convective fluxes in species' momentum balance laws. 

In either approach, the diffusion is a dissipative process in the sense that it is compatible with the entropy inequality. It contributes to the entropy balance law through an entropy production \cite{de1962non,krishna1997maxwell}. Nevertheless, other dissipative mechanisms could also be included in the analysis, viscous dissipation being the most prominent one. In the framework of irreversible thermodynamics, it is incorporated in the same manner as Fick diffusion -- by means of linear constitutive relation for the viscous stress tensor. However, in the case of Maxwell--Stefan model, such a generalization is less straightforward. It is usually included by assumption, and modeled by means of classical constitutive relations for Newtonian fluids \cite{kerkhof2005analysis,chen2015analysis}. 

The aim of this study is to include the viscous dissipation into the Maxwell--Stefan model by a systematic application of the moment method and diffusive scaling at the same time. The starting point is be the system of Boltzmann equations for non-reactive monatomic species. It is used to build up a corresponding system of moment equations, which consists of mass, momentum and momentum flux balance laws for the species, and the energy conservation law for the mixture. In fact, the momentum flux balance laws are the key ingredient which brings viscous dissipation into the model. Since our goal is to generalize the Maxwell--Stefan model of diffusion, all equations are transformed into a scaled (dimensionless) form, which inherits the already mentioned assumption on diffusion processes. The scaling formally ascribes a proper order of magnitude to each term appearing in the moment equations. Even more, it facilitates the systematic derivation of the higher-order model of diffusion, which inherits viscous dissipation, without the use of \emph{ad hoc} assumptions. 

The rest of the paper is organized as follows. Section \ref{sec:Kinetic_modelling} contains necessary information about kinetic modelling of mixtures, while Section \ref{sec:Diffusive_scaling} introduces the diffusive scaling into the modelling process. Section \ref{sec:MEP} is devoted to a derivation of the approximate velocity distribution function in the scaled form, which is needed for the closure of moment equations. Indeed, the usual \emph{ansatz} used for the classical Maxwell--Stefan model cannot be generalized in a straightforward way. Therefore, the maximum entropy principle is used as a tool for derivation of a proper approximation of the velocity distribution function in Theorem~
 \ref{th:MEP-Higher}. Section \ref{sec:Higher-order} contains the main results of the paper. The closed set of scaled higher-order moment equations are given in Theorem \ref{th:Moment-Equations} and Proposition \ref{th:Energy-conservation}, whereas the asymptotic limit of the higher-order model is given in Theorem \ref{th:Higher-order_M-S}. We also provide a compatibility condition which restricts the form of the equation of state in Proposition \ref{th:Ideal-gas}. Finally, the higher-order diffusion model is provided in Proposition \ref{th:Higher-order_Diffusion}. The paper ends with appropriate conclusions and outlook to possible further studies. 

\section{Kinetic modelling of mixtures} 
\label{sec:Kinetic_modelling}

In this section, the basics of mixture modelling in kinetic theory will be presented. This will serve as a reference for further study of the scaled equations and the maximum entropy principle. 

\subsection{Boltzmann equations for mixtures}

Consider a mixture of gases consisted of $S$ identifiable species. The state of each species is determined by the velocity distribution function $f^{i}(t,\vx,\vv) \equiv f^{i}(\vv)$, $i = 1, \ldots, S$. If the external forces are negligible, the behavior of a non-reactive mixture is described by the system of Boltzmann equations:
\begin{equation}\label{eq:Boltzmann-system}
  \partial_t f^i + \vv \cdot \nabla_{\vx} f^i = \sum_{j=1}^S Q^{ij}(f^i,f^j)(\vv),  \qquad 1\leq i \leq S,
\end{equation}
where for any $1\leq i,j\leq S$, $Q^{ij}(f^i,f^j)(\vv)$ is the collision operator which determines the rate of change of distribution functions due to particle collisions of species $i$ and $j$. It has the form:
\begin{equation}\label{eq:CollisionOp}
  Q^{ij}(f^i,f^j)(\vv) = \intR \intS \left[ f^{i}(\vv') f^{j}(\vv'_{\ast}) 
    - f^{i}(\vv) f^{j}(\vv_{\ast}) \right] \mathcal B^{ij}(\vv,\vv_{\ast},\boldsymbol{\sigma}) \dd \boldsymbol{\sigma} \dd \vv_{\ast},  
\end{equation}
where $\mathcal B^{ij}(\vv,\vv_{\ast},\boldsymbol{\sigma})$ are the collision cross sections. These cross sections are assumed for simplicity to correspond to Maxwell molecules, meaning that 
\begin{equation}\label{eq:Maxwell_molecules}
 \mathcal B^{ij} (\vv,\vv_*,\sigma) = b^{ij} (\cos \theta) .
\end{equation}
The collision operators can be written under a weak form as follows. For any~$\phi(\vv)$ such that the following integrals make sense, we have for any $1\leq i,j\leq S$
\begin{multline}\label{eq:weak_form}
\intR Q^{ij}(f^i,f^j)(\vv) \phi(\vv) \dd\vv \\
= \intR\intR\intS b^{ij}(\cos\theta) f^i(\vv) f^j(\vv_*) (\phi(\vv') - \phi(\vv)) \dd\sigma \dd\vv_*\dd\vv.
\end{multline}
Choosing $\phi(\vv) =1$, we obtain the obvious conservation property
\begin{equation}\label{eq:conservation_collision}
 \intR  Q^{ij}(f^i,f^j)(\vv) \dd\vv =0 .
\end{equation}

\subsection{Moment equations and moments} 

The moment method is one of the standard methods of analysis of the Boltzmann equations. It is based upon solving a finite number of transfer equations for the moments of the velocity distribution function, the so-called moment equations, instead of solving the Boltzmann equation itself. In other words, the solution of the Boltzmann equations is approximated by solving the corresponding set of moment equations. 

The moment equations are derived from the Boltzmann equations and actually correspond to their weak formulation. Starting from the Boltzmann equations \eqref{eq:Boltzmann-system}, after multiplication with a test function $\psi^{i}(\vv)$ and integration over velocity space, one obtains the set of moment equations for any $1\leq i \leq S$
\begin{equation}\label{eq:MomentEq}
  \partial_t \intR \psi^{i}(\vv) f^i \dd \vv 
	+ \nabla_{\vx} \cdot \intR \vv \psi^{i}(\vv) f^i \dd \vv
	= \sum_{j=1}^S \intR \psi^{i}(\vv) Q^{ij}(f^i,f^j)(\vv) \dd \vv.
\end{equation}
The moments of the velocity distribution function $ \intR \psi^{i}(\vv) f^i \dd \vv$ may correspond to the densities of observable macroscopic quantities, but they could also be tensorial quantities which may or may not have apparent physical meaning. 

Choice of the moments/densities which will be taken into account is crucial for the modelling and the accuracy of approximation. In a standard way, physically motivated moments for species $i$ are the partial mass $\rho^{i}$, momentum~$\rho^{i} u^{i}_{k}$, $1\leq k\leq 3$ and energy $\rho^{i} E^{i}$ densities 
\begin{align} \label{eq:Moments-Basic}
  \rho^{i}(t,\vx) & := \intR m_{i} f^{i} \dd \vv, 
  \nonumber \\ 
  \rho^{i}(t,\vx) u^{i}_{k}(t,\vx) & := \intR m_{i} v_{k} f^{i} \dd \vv, 
  \nonumber \\ 
  \rho^{i}(t,\vx) E^{i}(t,\vx) & := \intR \frac{1}{2} m_{i} \vert \vv \vert^2 f^{i} \dd \vv. 
\end{align}
For higher-order models, higher-order moments have to be taken as densities. To this end, for any species $i$, partial momentum fluxes $P^{i}_{k\ell}$ and their corresponding fluxes $P^{i}_{k\ell n}$, $1\leq k,\ell,n\leq 3$ will be introduced as follows
\begin{align} \label{eq:Moments-Higher}
  P^{i}_{k\ell}(t,\vx) & := \intR m_{i} v_{k} v_{\ell} f^{i} \dd \vv, 
  \nonumber \\ 
  P^{i}_{k\ell n}(t,\vx) & := \intR m_{i} v_{k} v_{\ell} v_{n} f^{i} \dd \vv.
\end{align}
By defining the peculiar velocities $c^{i}_{k} := v_{k} - u^{i}_{k}$ we may observe the following relations: 
\begin{gather} \label{eq:Intrinsic}
  \rho^{i} E^{i} = \frac{1}{2} \rho^{i} \vert \vu^{i}\vert^2 
	+ \rho^{i} \varepsilon^{i}, 
  \nonumber \\ 
  P^{i}_{k\ell} = \rho^{i} u^{i}_{k} u^{i}_{\ell} + p^{i}_{k\ell}, 
  \nonumber \\ 
  P^{i}_{k\ell n} = \rho^{i} u^{i}_{k} u^{i}_{\ell} u^{i}_{n} 
	+ u^{i}_{k} p^{i}_{\ell n} + u^{i}_{\ell} p^{i}_{nk} + u^{i}_{n} p^{i}_{k\ell}
	+ p^{i}_{k\ell n},
\end{gather}
where the partial pressures $p^{i}$, 
partial pressure tensors $p^{i}_{k\ell}$ and partial non-convective fluxes of the momentum fluxes are defined as follows: 
\begin{align} \label{eq:Moments-Peculiar}
  3 p^i(t,\vx) & := \intR  m_{i} \vert \vc^{i}\vert^2 f^{i} \dd \vv, 
  \nonumber \\ 
  p^{i}_{k\ell}(t,\vx) & := \intR m_{i} c^{i}_{k} c^{i}_{\ell} f^{i} \dd \vv, 
  \nonumber \\ 
  p^{i}_{k\ell n}(t,\vx) & := \intR m_{i} c^{i}_{k} c^{i}_{\ell} c^{i}_{n} f^{i} \dd \vv.
\end{align}
Note that the partial pressures $p^{i}$ are related to the trace of partial pressure tensors $p^{i}_{k\ell}$ by
\begin{equation} \label{eq:Partial-Pressure}
  p^{i} = \frac{1}{3} \sum_{k=1}^{3} p^{i}_{kk}.
\end{equation}
In the case of monatomic gases, we have the following relation between the partial internal energy densities and the partial pressures
$$3 p^{i} = 2 \rho^{i} \varepsilon^{i}.$$ 

For further computations, it will be useful to split the partial pressure tensors into a sum of spherical part, proportional to identity tensor, and deviatoric part $p^{i}_{\langle k\ell \rangle}$: 
\begin{equation}
  p^{i}_{k\ell} = p^{i} \delta_{k\ell} + p^{i}_{\langle k\ell \rangle}, 
\end{equation}
where $\delta_{k\ell}$ is the Kronecker delta. This has the obvious consequence that $\sum_{k=1}^{3} p^{i}_{\langle kk \rangle} = 0$.

\subsection{Kinetic approach to the Maxwell--Stefan model} 

Let us give a brief summary of  the existing results about the Maxwell--Stefan diffusion model, obtained from the moment equations \eqref{eq:MomentEq}.
To derive the \emph{classical Maxwell--Stefan diffusion model}, one has to use moment equations for the mass densities $\rho^{i}$ and for the momentum densities $\rho^{i} u^{i}_{k}$ \cite{boudin2015maxwell}. For derivation of the \emph{non-isothermal Maxwell--Stefan model}, with or without chemical reactions, apart from moment equations for mass and momentum densities, one has to include the moment equations for energy densities $\rho^{i} E^{i}$ \cite{hutridurga2018existence,anwasia2020formal,anwasia2020maxwell}. More precisely, in the standard approximation (the so-called asymptotic limit, to be presented in the sequel), it is sufficient to use the energy conservation equation for the whole mixture. 

All the approximations mentioned above assume that the velocity distribution functions are in the form of a Maxwellian, i.e. in some kind of local equilibrium in which dissipative effects (viscous and thermal) are neglected. However, if one takes them into account, higher-order moments and their corresponding transfer equations have to be exploited. Our study of \emph{higher-order Maxwell--Stefan models} will be limited to the extension which takes into account only viscous dissipation through moment equations for partial momentum fluxes $P^{i}_{k\ell}$.

\section{Diffusive scaling}
\label{sec:Diffusive_scaling}

In the analysis of particular processes, not all the terms in the governing equations have equal importance, i.e. some of them may be neglected. Such simplifications require insight into the features of the process which is analyzed and appropriate order-of-magnitude estimates. A usual way to reach that goal is to put the governing equations into dimensionless/scaled form. This leads to a reduction of the number of parameters in the model, and yields an  estimate of the order of magnitude of particular terms. 

Diffusion processes in gaseous mixtures possess two important features: (i) they occur in the hydrodynamic setting, in which the characteristic macroscopic length scale is much larger than the mean free path of the particles, and (ii) the characteristic macroscopic velocity is much smaller than the reference molecular velocity, which is of the order of the speed of sound. A scaling of the governing equations which reflects these features is called diffusive scaling. 

\subsection{Boltzmann equations in diffusive scaling} 

The dimensionless form of the Boltzmann equations~\eqref{eq:Boltzmann-system} is obtained by means of scaling  variables similarly as in~\cite{anwasia2022maximum}. To this end, the macroscopic time and length scales are introduced, denoted $\tau$ and $L$ respectively, as well as the reference kinetic temperature $T_{0}$. They imply two independent velocity scales: $u_{0} = \frac{L}{\tau}$, which is the speed of macroscopic transport of the gas over distance $L$ in time $\tau$, and $c_{0} = \left( \frac{5}{3} \frac{k_{B}}{m_{0}} T_{0} \right)^{1/2}$, which is the speed of sound in a monatomic gas, with $m_{0}$ being the average atomic mass of the mixture and $k_B$ the Boltzmann constant. Using $\tau$, $L$ and $c_{0}$ as reference time, space and velocity scales, respectively, one may introduce the following dimensionless quantities: 
\begin{gather*}
  \hat{t} = \frac{t}{\tau}, \quad \hat{\vx} = \frac{\vx}{L}, \quad 
  \hat{\vv} = \frac{\vv}{c_{0}}, \quad 
  \hat{f}^{i}(\hat{\vv}) = \frac{L^{3} c_{0}^{3}}{N} f^{i}(\vv), 
  \\ 
  \hat{B}^{ij}(\hat{\vv}, \hat{\vv}_{\ast}, \hat{\boldsymbol{\sigma}})
    = \frac{1}{c_{0} 4 \pi r^{2}} B^{ij}(\vv,\vv_{\ast},\boldsymbol{\sigma}), 
    \\
  \hat{Q}^{ij}(\hat{f}^i,\hat{f}^j)(\hat{\vv}) 
    = \left( \frac{L^{3} c_{0}^{3}}{N} \right)^{2} \frac{1}{c_{0} 4 \pi r^{2}} 
      \frac{1}{c_{0}^{3}} Q^{ij}(f^i,f^j)(\vv),
\end{gather*}
where $N$ is the number of gas molecules in a volume $L^{3}$ and $r$ is the average radius of the molecules. In such a way, equations \eqref{eq:Boltzmann-system} acquire the following dimensionless form: 
\begin{equation}\label{eq:Boltzmann-DimLess}
  \mathrm{Ma} \partial_{\hat{t}} \hat{f}^i + \hat{\vv} \cdot \nabla_{\hat{\vx}} \hat{f}^i 
    = \frac{1}{\mathrm{Kn}} \sum_{j=1}^S \hat{Q}^{ij}(\hat{f}^i,\hat{f}^j)(\hat{\vv}), 
\end{equation}
where
\begin{equation*}
  \mathrm{Ma} = \frac{u_{0}}{c_{0}}, \quad 
  \mathrm{Kn} = \frac{\text{mean free path}}{\text{macroscopic length scale}}
    = \frac{L^{3}}{N \times 4 \pi r^{2}} \times \frac{1}{L} = \frac{L^2}{4\pi r^2 N},
\end{equation*}
are the Mach and the Knudsen numbers, respectively. 

In this paper, we focus on the diffusive scaling, which is a particular form of dimensionless Boltzmann equations \eqref{eq:Boltzmann-DimLess} in which the Mach and Knudsen numbers are assumed to be of the same small order of magnitude
\begin{equation}\label{eq:DiffusiveScaling}
  \mathrm{Ma} = \mathrm{Kn} = \alpha \ll 1. 
\end{equation} 
With this assumption, we obtain the Boltzmann equations in diffusive scaling:
\begin{equation}\label{eq:Boltzmann-Diffusive}
  \alpha \partial_t f^i + \vv \cdot \nabla_{\vx} f^i 
    = \frac{1}{\alpha} \sum_{j=1}^S Q^{ij}(f^i,f^j)(\vv), 
\end{equation}
where we now dropped the hats for readability. The factor $1/\alpha$ on the right-hand side is typical for a hydrodynamic limit of the Boltzmann equations, while multiplication of the time derivative with $\alpha$ implies that the processes are slow. 

\subsection{Moment equations in diffusive scaling} 

Starting from the Boltzmann equations in the diffusive scaling \eqref{eq:Boltzmann-Diffusive}, after multiplication with a test function $\psi^{i}(\vv)$ and integration over the velocity space, one obtains the (dimensionless) moment equations in diffusive scaling:
\begin{equation}\label{eq:MomentEq-Diffusive}
  \alpha \partial_t \intR \psi^{i}(\vv) f^i \dd \vv 
    + \nabla_{\vx} \cdot \intR \vv \psi^{i}(\vv) f^i \dd \vv
	= \frac{1}{\alpha} \sum_{j=1}^S \intR \psi^{i}(\vv) Q^{ij}(f^i,f^j)(\vv) \dd \vv.
\end{equation}
The choice of the moments, i.e. of the test functions, determines the state space on which the Boltzmann equations are projected and their solution is approximated. To close the system of moment equations \eqref{eq:MomentEq-Diffusive} one has to approximate the velocity distribution function. Grad \cite{grad1949kinetic} originally used Hermite polynomial expansions as an approximation, but it was shown by Kogan \cite{kogan1969rarefied} that equivalent results may be obtained by means of a variational approach, the so-called maximum entropy principle.

\section{The maximum entropy principle}
\label{sec:MEP}

The maximum entropy principle (MEP) is a constrained variational formulation which determines the approximate velocity distribution \cite{muller1993extended,dreyer1987maximisation,Kremer-Introduction}. The functional to be maximized is the kinetic entropy $H(t,\vx)$, defined as: 
\begin{equation} \label{eq:Entropy}
  H(t,\vx) := \sum_{i=1}^{S} H^{i}(t,\vx), \quad 
  H^{i}(t,\vx) := - k_{\mathrm{B}} \intR f^{i} \log (b^{i} f^{i}) \dd \vv, 
\end{equation}
where $b^{i}$ is a dimensional constant used to make dimensionless the argument of the $\log$ function.
The constraints are taken to be the moments/macroscopic variables which determine the state space. To derive the Maxwell--Stefan model of diffusion it is sufficient to choose the partial mass, momentum and energy densities as constraints. This implies the approximate velocity distribution functions to be the local Maxwellians. In this study we extend the usual state space to include the pressure tensor in it. Such an extension will provide us with a more detailed insight into dissipation, other than diffusion, which may occur in the system. 

In the sequel we shall follow the procedure given in \cite{anwasia2022maximum} and exploit the diffusive scaling to derive properly scaled velocity distribution functions $\hat{f}^{i}$. To that end we introduce the following dimensionless (scaled) variables: 
\begin{gather}\label{eq:DLess-Macroscopic}
  \hat{H}^{i} = \frac{L^{3}}{k_{\mathrm{B}} N} H^{i}, \quad 
  \hat{b}^{i} = \frac{N}{L^{3} c_{0}^{3}} b^{i}, \quad 
  \hat{m}_{i} = \frac{m_{i}}{m_{0}}, \quad 
  \hat{\rho}^{i} = \frac{L^{3}}{m_{0} N} \rho^{i}, \quad 
  \hat{\vu}^{i} = \frac{\vu^{i}}{u_{0}}, 
  \nonumber \\
  \hat{\vc}^{i} = \frac{\vc^{i}}{c_{0}}, \quad 
  \hat{E}^{i} = \frac{E^{i}}{c_{0}^{2}}, \quad 
  \hat{T} = \frac{T}{T_{0}}, 
  \nonumber \\
  \hat{P}^{i}_{k\ell} = \frac{L^{3}}{m_{0} N c_{0}^{2}} P^{i}_{k\ell}, \quad
  \hat{p}^{i}_{k\ell} = \frac{L^{3}}{m_{0} N c_{0}^{2}} p^{i}_{k\ell}, \quad
  \hat{P}^{i}_{k\ell n} = \frac{L^{3}}{m_{0} N c_{0}^{3}} P^{i}_{k\ell n}, 
\end{gather}
and we again drop the hats in further computations for convenience. With these, equations \eqref{eq:Moments-Basic} and \eqref{eq:Moments-Higher} are the same, except for the substitution $u^{i}_{k} \mapsto \alpha u^{i}_{k}$, whereas the scaled kinetic entropy \eqref{eq:Entropy} reads:
\begin{equation} \label{eq:Entropy-DLess}
  H(t,\vx) = \sum_{i=1}^{S} H^{i}(t,\vx), \quad 
  H^{i}(t,\vx) = - \intR f^{i} \log (b^{i} f^{i}) \dd \vv.
\end{equation}
The MEP can now be formulated as a variational problem with constraints in diffusive scaling. 

For the completeness of the study we shall firstly recover the local equilibrium velocity distribution functions, and then derive their higher-order approximations which comprise the partial pressure tensors.

\subsection{The local equilibrium approximation} 

The MEP for the local equilibrium approximation consists in finding the velocity distribution functions $f^{i}(t,\vx,\vv)$ maximizing the kinetic entropy \eqref{eq:Entropy}
subject to the following constraints: 
\begin{align} \label{eq:MomentsLocal-DLess}
  \rho^{i} & = \intR m_{i} f^{i} \dd \vv, 
  \nonumber \\ 
  \alpha \rho^{i} u^{i}_{k} & = \intR m_{i} v_{k} f^{i} \dd \vv, 
  \nonumber \\ 
  \alpha^{2} \rho^{i} \vert \vu^{i}\vert^2 +2 \rho^i \varepsilon^i 
    & = \intR m_{i} \vert \vv\vert^2 f^{i} \dd \vv.
\end{align} 
In \eqref{eq:MomentsLocal-DLess} we dropped $(t,\vx)$ dependence of macroscopic variables. 

\begin{theorem} \label{th:MEP-Local}
The velocity distribution functions which maximize the entropy \eqref{eq:Entropy}  
with constraints \eqref{eq:MomentsLocal-DLess} have the following form: 
\begin{equation} \label{eq:fi-Local0}
  f^{i}(t,\vx,\vv) = \frac{\rho^{i}}{m_{i}} \left( \frac3{4\pi \varepsilon^i}
   \right)^{3/2} 
    \exp \left(-  \frac{3\vert \vv - \alpha \vu^{i} \vert^{2}}{4\varepsilon^i} 
    \right).
\end{equation}
\end{theorem}

\begin{proof} 
Although the proof follows the same steps as in \cite{anwasia2022maximum}, we shall briefly repeat them.  
Let us define the following extended functional 
\begin{multline*}
 \intR \mathcal{L} \left(\vv, f^{i}, \nabla_{\vv} f^{i} \right) \dd \vv :=\\
 \intR \sum_{i=1}^{S} \left( f^{i} \log(b^{i} f^{i}) 
    + \lambda_{\rho}^{i} m_{i} f^{i} 
    + \sum_{k=1}^{3} \lambda_{u_{k}}^{i} m_{i} v_{k} f^{i} 
    + \lambda_\varepsilon^{i} m_{i} \vert \vv \vert^{2} f^{i} \right) \dd \vv ,
\end{multline*}  
where $\lambda_{\rho}^{i}(t,\vx), \lambda_{u_{k}}^{i}(t,\vx), \lambda_\varepsilon^{i}(t,\vx) \in \mathbb{R}$ are the unknown multipliers. Taking into account that $\mathcal{L}$ is independent of $\nabla_{\vv} f^{i}$, the necessary condition for extremum reduces to $\partial \mathcal{L}/\partial f^{i} = 0$, $i = 1, \ldots, S$, yielding $S$ uncoupled equations whose solutions read for any $1\leq i \leq S$
\begin{equation*}
  f^{i} = \frac{1}{b^{i}} \exp \left[ - \left( 1 + m_{i} \lambda_{\rho}^{i} 
    + m_{i} \sum_{k=1}^{3} \lambda_{u_{k}}^{i} v_{k} 
    + m_{i} \lambda_\varepsilon^{i} \vert \vv \vert^{2} \right) \right].
\end{equation*}
By inserting these results into the constraints \eqref{eq:MomentsLocal-DLess} and with rather straightforward computations the following relations are obtained: 
\begin{gather*}
  \lambda_\varepsilon^{i} =\frac3{4m_i \varepsilon^i}, 
  \quad 
  \frac{\lambda_{u_{k}}^{i}}{2 \lambda_\varepsilon^{i}} = - \alpha u_{k}^{i}, 
  \\
  \frac{m_{i}}{b^{i}} \exp \left( - 1 - \lambda_{\rho}^{i} \right) 
    \exp \left( \frac{m_{i}}{4 \lambda_\varepsilon^{i}} 
      \sum_{k=1}^{3} \left( \lambda_{u_{k}}^{i} \right)^{2} 
      \right)
    \left( \frac{\pi}{m_{i} \lambda_\varepsilon^{i}} \right)^{3/2} = \rho^{i}.
\end{gather*}
Returning these expressions into the form of $f^i$, one obtains \eqref{eq:fi-Local0}, which completes the proof. 
\end{proof}

\subsection{The higher-order approximation} 

Higher-order approximations of the velocity distribution functions aim at extending the state space and capturing non-equilibrium effects. To that end, the kinetic entropy \eqref{eq:Entropy} remains the same, whereas the constraints have the following form:
\begin{align}
  \rho^{i} & = \intR m_{i} f^{i} \dd \vv, 
  \label{eq:MomentsHigher-DLess0} \\ 
  \alpha \rho^{i} u^{i}_{k} & = \intR m_{i} v_{k} f^{i} \dd \vv, 
  \label{eq:MomentsHigher-DLess1} \\ 
  \alpha^{2} \rho^{i} u^{i}_{k} u^{i}_{\ell} + p^{i}_{k\ell} 
	& = \intR m_{i} v_{k} v_{\ell} f^{i} \dd \vv. \label{eq:MomentsHigher-DLess2}
\end{align} 
In the sequel it will be more convenient to use the peculiar velocities $c^{i}_{k} = v_{k} - \alpha u^{i}_{k}$ and formulate the MEP for the constraints:
\begin{align} \label{eq:MomentsPeculiar-DLess}
  \rho^{i} & = \intR m_{i} f^{i} \dd \vc, 
  \nonumber \\ 
  0 & = \intR m_{i} c^{i}_{k} f^{i} \dd \vc, 
  \nonumber \\ 
  p^{i}_{k\ell} & = \intR m_{i} c^{i}_{k} c^{i}_{\ell} f^{i} \dd \vc.
\end{align}
Note that $\dd \vv = \dd \vc$ holds. 

The extended functional for the constraints \eqref{eq:MomentsHigher-DLess0}-\eqref{eq:MomentsHigher-DLess2} is 
\begin{multline} \label{eq:MEP-Extended}
    \intR \sum_{i=1}^{S} \left( f^{i} \log(b^{i} f^{i}) 
	+ \lambda_{\rho}^{i} m_{i} f^{i} 
	+ \sum_{k=1}^{3} \lambda_{u_{k}}^{i} m_{i} v_{k} f^{i} 
	\right.
   \\
     \left.
	\quad + \sum_{k,\ell=1}^{3} \lambda_{P_{k\ell}}^{i} m_{i} v_{k} v_{\ell} f^{i} \right) \dd \vv. 
\end{multline} 
The same extended functional written for $u^{i}_{k} = 0$, $i = 1, \ldots, S$, corresponds to MEP for the constraints \eqref{eq:MomentsPeculiar-DLess}:
\begin{align} \label{eq:MEP-Extended_Peculiar}
  \mathcal{H} & = \intR \sum_{i=1}^{S} \left( f^{i} \log(b^{i} f^{i}) 
	+ \tilde{\lambda}_{\rho}^{i} m_{i} f^{i} 
	+ \sum_{k=1}^{3} \tilde{\lambda}_{u_{k}}^{i} m_{i} c^{i}_{k} f^{i} 
	\right.
  \nonumber \\
  & \left.
	\quad + \sum_{k,\ell=1}^{3} \tilde{\lambda}_{p_{k\ell}}^{i} m_{i} c^{i}_{k} c^{i}_{\ell} 
	f^{i} \right) \dd \vc, 
\end{align} 
but with different multipliers. However, since \eqref{eq:MEP-Extended} and \eqref{eq:MEP-Extended_Peculiar} must have the same value due to Galilean invariance, comparison leads to the following relations between the multipliers: 
\begin{align} \label{eq:Multipliers}
  \lambda_{\rho}^{i} & = \tilde{\lambda}_{\rho}^{i} 
    - \alpha \sum_{k=1}^{3} \tilde{\lambda}_{u_{k}}^{i} u^{i}_{k} 
    + \alpha^{2} \sum_{k,\ell=1}^{3} \tilde{\lambda}_{p_{k\ell}}^{i} u^{i}_{k} u^{i}_{\ell}, 
  \nonumber \\ 
  \lambda_{u_{k}}^{i} & = \tilde{\lambda}_{u_{k}}^{i} 
    - 2 \alpha \sum_{\ell=1}^{3} \tilde{\lambda}_{p_{k\ell}}^{i} u^{i}_{\ell}, 
  \nonumber \\ 
  \lambda_{P_{k\ell}}^{i} & = \tilde{\lambda}_{p_{k\ell}}^{i} = \tilde{\lambda}_{p_{\ell k}}^{i}.
\end{align}
Therefore, the whole analysis of the higher-order approximation can be based upon the MEP for the constraints \eqref{eq:MomentsPeculiar-DLess}, and expressed afterwards in terms of \eqref{eq:MomentsHigher-DLess0}-\eqref{eq:MomentsHigher-DLess2} if needed. 

\begin{theorem} \label{th:MEP-Higher}
The velocity distribution functions which maximizes the entropy \eqref{eq:Entropy} 
with constraints \eqref{eq:MomentsPeculiar-DLess} have the following form: 
\begin{equation} \label{eq:fi-Local}
  f^{i}(t,\vx,\vc) = \frac{\rho^{i}}{m_{i}} \left( \frac{\rho^{i}}{2 \pi} \right)^{3/2} 
    \frac{1}{(\det \tp^{i})^{1/2}} 
	\exp \left( - \frac{\rho^{i}}{2} {\vc^{i}}^T (\tp^{i})^{-1} \vc^{i} \right),
\end{equation}
where $\tp^{i} = \left\{ p^{i}_{k\ell} \right\}_{k,\ell=1}^{3}$ is the partial pressure tensor, and superscript ${}^T$ stands for the transpose of the vector. 
\end{theorem} 

\begin{proof} 

The strategy of the proof is the same as in Theorem \ref{th:MEP-Local}. Here, it will be based upon the extended functional \eqref{eq:MEP-Extended_Peculiar} which we shall briefly denote as: 
\begin{equation*}
  \mathcal{H} = \intR \mathcal{L}\left( \vc, f^{i}, \nabla_{\vc}f^{i} \right) \dd \vc. 
\end{equation*}
Since $\mathcal{L}$ is independent of $\nabla_{\vc}f^{i}$, the necessary condition for extremum reduces to $\partial \mathcal{L}/\partial f^{i} = 0$, $i = 1, \ldots, S$, implying $S$ uncoupled equations whose solutions are 
\begin{equation} \label{eq:fi10-General}
  f^{i} = \frac{1}{b^{i}} \exp \left[ - \left( 1 + m_{i} \tilde{\lambda}_{\rho}^{i} 
    + m_{i} \sum_{k=1}^{3} \tilde{\lambda}_{u_{k}}^{i} c_{k}^{i} 
    + m_{i} \sum_{k,\ell=1}^{3} \tilde{\lambda}_{p_{k\ell}}^{i} c_{k}^{i} c_{\ell}^{i} 
    \right) \right]. 
\end{equation}
To determine the multipliers, we have to compute the moments \eqref{eq:MomentsPeculiar-DLess}, but their direct computation using \eqref{eq:fi10-General} is cumbersome. Therefore, the proof will be completed in two steps. 

First, let us prove that $\tilde{\lambda}_{u_{k}}^{i} = 0$, $i = 1, \ldots, S$, $k = 1, 2, 3$. To this end, we shall first rewrite \eqref{eq:fi10-General} as
\begin{equation} \label{eq:fi10-Direct}
  f^{i} = \frac{1}{b^{i}} \exp \left[ - \left( 1 + m_{i} \tilde{\lambda}_{\rho}^{i} 
    + m_{i} \tilde{\boldsymbol{\lambda}}_{\vu}^{iT} \vc^{i} 
    + m_{i} \vc^{iT} \tilde{\boldsymbol{\lambda}}_{\tp}^{i} \vc^{i} \right) \right], 
\end{equation}
where $\tilde{\boldsymbol{\lambda}}_{\vu}^{i} = \left\{ \tilde{\lambda}_{u_{k}}^{i} \right\}_{k=1}^{3}$ is the vector of momentum multipliers and $\tilde{\boldsymbol{\lambda}}_{\tp}^{i} = \left\{ \tilde{\lambda}_{p_{k\ell}}^{i} \right\}_{k,\ell=1}^{3}$ is the matrix of pressure tensor multipliers. Since $\tilde{\boldsymbol{\lambda}}_{\mathbf{p}}^{i}$ is a real symmetric matrix, it can be diagonalized by means of an orthogonal matrix $\mathbf{R}$ satisfying $\mathbf{R}^{T} \mathbf{R} = \mathbf{R} \mathbf{R}^{T}= \mathbf{I} $. Let us introduce the orthogonal transformation in the space of peculiar velocities
\begin{equation} \label{eq:OrthogonalTransformation}
  \mathbf{C}^{i} = \mathbf{R}^{i} \vc^{i} 
  \quad \Rightarrow \quad
  \vc^{i} = \mathbf{R}^{iT} \mathbf{C}^{i}. 
\end{equation}
This implies
\begin{equation} \label{eq:Transformed-CC}
  \vc^{iT} \tilde{\boldsymbol{\lambda}}_{\tp}^{i} \vc^{i} 
    = \mathbf{C}^{iT} \mathbf{R}^{i} \tilde{\boldsymbol{\lambda}}_{\mathbf{p}}^{i} 
      \mathbf{R}^{iT} \mathbf{C}^{i} 
    = \mathbf{C}^{iT} \tilde{\boldsymbol{\Lambda}}_{\tp}^{i} \mathbf{C}^{i} 
    = \sum_{k=1}^{3} \tilde{\Lambda}_{k}^{i} (C_{k}^{i})^{2}, 
\end{equation}
where we introduced the transformed diagonal matrix of multipliers
\begin{equation*}
  \tilde{\boldsymbol{\Lambda}}_{\tp}^{i} 
    = \mathbf{R}^{i} \tilde{\boldsymbol{\lambda}}_{\tp}^{i} \mathbf{R}^{iT} 
    = \operatorname{diag} \left\{ \tilde{\Lambda}_{k}^{i} \right\}_{k=1}^{3}. 
\end{equation*}
Further, by means of same transformation, we obtain
\begin{equation} \label{eq:Transformed-C}
  \tilde{\boldsymbol{\lambda}}_{\vu}^{iT} \vc^{i} 
    = \tilde{\boldsymbol{\lambda}}_{\vu}^{iT} \mathbf{R}^{iT} \mathbf{C}^{i} 
    = \tilde{\mathbf{L}}_{\vu}^{iT} \mathbf{C}^{i}
    = \sum_{k=1}^{3} \tilde{L}_{k}^{i} C_{k}^{i}, 
\end{equation}
where we introduced the transformed vector of multipliers  
\begin{equation} \label{eq:VelMultipliers-Transformed}
  \tilde{\mathbf{L}}_{\vu}^{i} 
    = \mathbf{R}^{i} \tilde{\boldsymbol{\lambda}}_{\vu}^{i} 
    = \left\{ \tilde{L}_{k}^{i} \right\}_{k=1}^{3}.
\end{equation}
Taking into account \eqref{eq:Transformed-CC} and \eqref{eq:Transformed-C}, the velocity distributions \eqref{eq:fi10-Direct} may be finally transformed into
\begin{align} \label{eq:fi10-Transformed}
  f^{i} & = \frac{1}{b^{i}} \exp \left[ - \left( 1 + m_{i} \tilde{\lambda}_{\rho}^{i} 
    + m_{i} \sum_{k=1}^{3} \tilde{L}_{k}^{i} C_{k}^{i} 
    + m_{i} \sum_{k=1}^{3} \tilde{\Lambda}_{k}^{i} (C_{k}^{i})^{2} \right) \right] 
  \nonumber \\
  & = \frac{1}{b^{i}} \exp \left[ - \left( 1 + m_{i} \tilde{\lambda}_{\rho}^{i} \right) \right] 
    \prod_{k=1}^{3} \exp \left[ - m_{i} \left( \tilde{L}_{k}^{i} C_{k}^{i} 
    + \tilde{\Lambda}_{k}^{i} (C_{k}^{i})^{2} \right) \right]. 
\end{align}

For the computation of constraints \eqref{eq:MomentsPeculiar-DLess} we have to take into account that $\dd \vc^{i} = \dd \mathbf{C}^{i}$, since Jacobian of the transformation \eqref{eq:OrthogonalTransformation} is $\vert \mathbf{J} \vert = \vert \mathbf{R}^{iT} \vert = 1$. In what follows, we shall compute the constraints \eqref{eq:MomentsPeculiar-DLess}$_{2}$. Applying \eqref{eq:OrthogonalTransformation}$_{2}$ in component form, $c_{n}^{i} = \sum_{j=1}^{3} R_{jn}^{i} C_{j}^{i}$, and taking  advantage of \eqref{eq:fi10-Transformed} which facilitates a simple application of Fubini's theorem, after straightforward computation we obtain 
\begin{align*} \label{eq:VelConstraint-Transformed}
  0 & = - \frac{m_{i}}{b^{i}} \exp \left[ - \left( 1 + m_{i} \tilde{\lambda}_{\rho}^{i} 
    \right) \right] \left( \frac{\pi}{m_{i}} \right)^{3/2} \left( \sum_{j=1}^{3} 
    R_{jn}^{i} \frac{\tilde{L}_{j}^{i}}{2 \tilde{\Lambda}_{j}^{i}} \right) 
  \nonumber \\
  & \quad \times \prod_{k=1}^{3} \frac{1}{\sqrt{\tilde{\Lambda}_{k}^{i}}} \exp \left[ 
    \frac{m_{i} (\tilde{L}_{k}^{i})^{2}}{4 \tilde{\Lambda}_{k}^{i}} \right].
\end{align*}
The only possibility to satisfy these constraints is to impose
\begin{equation*}
  \sum_{j=1}^{3} R_{jn}^{i} \frac{\tilde{L}_{j}^{i}}{2 \tilde{\Lambda}_{j}^{i}} = 0,
\end{equation*}
which amounts to a homogeneous system of linear algebraic equations. Since $\vert \mathbf{R}^{iT} \vert = \vert \mathbf{R}^{i} \vert = 1$, the matrix of coefficients is non-singular and there exists only a trivial solution
\begin{equation} \label{eq:VelMultipliers-Solution}
  \frac{\tilde{L}_{j}^{i}}{2 \tilde{\Lambda}_{j}^{i}} = 0 
  \quad \Rightarrow \quad 
  \tilde{L}_{j}^{i} = 0. 
\end{equation}
Using \eqref{eq:VelMultipliers-Transformed} and \eqref{eq:VelMultipliers-Solution}, by applying the same regularity arguments we arrive to
\begin{equation*}
  \tilde{\mathbf{L}}_{\vu}^{i} 
    = \mathbf{R}^{i} \tilde{\boldsymbol{\lambda}}_{\vu}^{i} 
    = \mathbf{0} 
  \quad \Rightarrow \quad 
  \tilde{\boldsymbol{\lambda}}_{\vu}^{i} = \mathbf{0},
\end{equation*}
which is equivalent to $\tilde{\lambda}_{u_{k}}^{i} = 0$. This result implies a simplified form of the velocity distribution functions \eqref{eq:fi10-General}
\begin{equation} \label{eq:fi10-Reduced}
  f^{i} = \frac{1}{b^{i}} \exp \left[ - \left( 1 + m_{i} \tilde{\lambda}_{\rho}^{i} 
	+ m_{i} \sum_{k,\ell=1}^{3} \tilde{\lambda}_{p_{k\ell}}^{i} c_{k}^{i} c_{\ell}^{i} 
	\right) \right]. 
\end{equation}

The second step consists in computation of the constraints \eqref{eq:MomentsPeculiar-DLess}$_{1}$ and \eqref{eq:MomentsPeculiar-DLess}$_{3}$ using \eqref{eq:fi10-Reduced}, which leads to
\begin{align} 
  \rho^{i} & = \frac{m_{i}}{b^{i}} \exp \left( - 1 - m_{i} \tilde{\lambda}_{\rho}^{i} \right) 
    \left( \frac{\pi}{m_{i}} \right)^{3/2} 
    \left( \det \tilde{\boldsymbol{\lambda}}_{\tp}^{i} \right)^{-1/2}, 
  \label{eq:fi10-rho} \\
  p_{k\ell}^{i} & = \frac{1}{2 b^{i}} \exp \left( - 1 - m_{i} \tilde{\lambda}_{\rho}^{i} \right) 
    \left( \frac{\pi}{m_{i}} \right)^{3/2} 
    \operatorname{adj}\left( \tilde{\boldsymbol{\lambda}}_{\tp}^{i} \right)_{k\ell}
    \left( \det \tilde{\boldsymbol{\lambda}}_{\tp}^{i} \right)^{-3/2}, 
  \label{eq:fi10-p_kl}
\end{align}
where $\operatorname{adj}\left( \mathbf{A} \right)$ denotes the classical adjoint of a matrix $\mathbf{A}$. Using \eqref{eq:fi10-rho} and taking into account that
\begin{equation*}
  \operatorname{adj}\left( \tilde{\boldsymbol{\lambda}}_{\tp}^{i} \right)_{k\ell}
    \left( \det \tilde{\boldsymbol{\lambda}}_{\tp}^{i} \right)^{-1} 
    = \left( \tilde{\boldsymbol{\lambda}}_{\tp}^{i} \right)^{-1}_{k\ell}, 
\end{equation*}
we can reduce \eqref{eq:fi10-p_kl} to
\begin{equation}
  p_{k\ell}^{i} = \frac{1}{2} \frac{\rho^{i}}{m_{i}} 
    \left( \tilde{\boldsymbol{\lambda}}_{\tp}^{i} \right)^{-1}_{k\ell} 
  \quad \Leftrightarrow \quad 
  \tp^{i} = \frac{1}{2} \frac{\rho^{i}}{m_{i}} 
    \left( \tilde{\boldsymbol{\lambda}}_{\tp}^{i} \right)^{-1},
\end{equation}
or, equivalently
\begin{equation} \label{eq:fi10-lambda_p}
  \tilde{\boldsymbol{\lambda}}_{\tp}^{i} = \frac{1}{2} \frac{\rho^{i}}{m_{i}} 
    \left( \tp^{i} \right)^{-1}. 
\end{equation} 
Using \eqref{eq:fi10-rho} and \eqref{eq:fi10-lambda_p} in \eqref{eq:fi10-Reduced}, after some direct computations one easily obtains \eqref{eq:fi-Local}, which completes the proof.
\end{proof} 

\begin{remark}
The result of Theorem \ref{th:MEP-Higher} presents a generalization of the 10 moments approximation of the velocity distribution function given in \cite{levermore1996moment} to the case of mixtures. We here provide the computations in the scaled form with a more detailed proof. 
\end{remark}

Theorem \ref{th:MEP-Higher} implies a simpler structure of the flux of momentum fluxes $P^{i}_{k\ell n}$. Indeed, if the velocity distribution function $f^i$ has the form \eqref{eq:fi-Local}, since $p^i_{k\ell n}$ is given by \eqref{eq:Moments-Peculiar}$_3$, a simple parity argument implies that $p^{i}_{k\ell n} = 0$. We thus have the following corollary.
\begin{corollary}
For the velocity distribution function determined by the higher-order approximation \eqref{eq:fi-Local}, 
 the scaled total flux of the momentum fluxes reads: 
\begin{equation} \label{eq:moments3}
  P^{i}_{k\ell n} = \alpha^{3} \rho^{i} u^{i}_{k} u^{i}_{\ell} u^{i}_{n} 
	+ \alpha \left( u^{i}_{k} p^{i}_{\ell n} + u^{i}_{\ell} p^{i}_{nk} 
	+ u^{i}_{n} p^{i}_{k\ell} \right).
\end{equation}
\end{corollary}

\section{Higher-order Maxwell--Stefan model}
\label{sec:Higher-order}

In what follows, we shall derive the higher-order Maxwell--Stefan model of diffusion. It will be based upon higher-order moment equations in dimensionless form, keeping up the scaling provided by the general form \eqref{eq:MomentEq-Diffusive}. They will provide a clear order of magnitude estimate for all the ingredients in moment equations, and thus make further approximations formally consistent. Our aim is to include viscous dissipation by means of moment equations for the momentum fluxes, and analyze different levels of approximation which may be used to describe diffusion processes. 

\subsection{The moment equations}

The set of moment equations which we need for generalization of the Maxwell--Stefan model consists of the mass, momentum, momentum flux and energy balance laws for the species (Theorem \ref{th:Moment-Equations}). In the so-called asymptotic limit, relevant to the derivation of Maxwell--Stefan equations in usual form, we shall need only the energy conservation law for the mixture (Proposition \ref{th:Energy-conservation}), rather than balance laws for each species. 

\begin{theorem} \label{th:Moment-Equations}
The Boltzmann equations \eqref{eq:Boltzmann-Diffusive} formally lead to the following set of moment equations, consisting for each species in
\begin{itemize}
 \item the mass conservation equations
\begin{equation}\label{eq:moment_mass}
\alpha\Big(\partial_t\rho^i + \sum_{k=1}^3 \partial_{x_k} ( \rho^i u^i_k ) \Big)= 0, \qquad  1\leq i \leq S.
\end{equation}
\item the momentum balance laws
\begin{multline}\label{eq:moment_momentum}
 \alpha^2\left[ \partial_t (\rho^i u^i_\ell) +\sum_{k=1}^3 \partial_{x_k} ( \rho^i u^i_k u^i_\ell) \right] +\sum_{k=1}^3 \partial_{x_k} p^i_{k\ell} \\
  =\sum_{j=1}^S   \frac{2\pi \|b^{ij}\|_{L^1}}{m_i+m_j} \rho^i \rho^j (u^j_\ell - u^i_\ell), \qquad  1\leq \ell\leq 3,~1\leq i \leq S.
\end{multline}
\item the momentum flux balance laws
\begin{multline}\label{eq:moment_pressure}
 \partial_t (\alpha^3 \rho^i u^i_k u^i_\ell+\alpha p^i_{k\ell}) \\
 +\sum_{n=1}^3 \partial_{x_n} (\alpha^3 \rho^i u^i_k u^i_\ell u^i_n+ \alpha (u^i_k p^i_{\ell n} + u^i_\ell p^i_{nk} + u^i_n p^i_{k\ell})) \\
 = \sum_{j=1}^S\Bigg(\frac{2\pi \|b^{ij}\|_{L^1}}{(m_i+m_j)^2} \Big\{ -(2m_i+m_j)\rho^j(\alpha\rho^i u^i_k u^i_\ell + \frac1\alpha p^i_{k\ell})\\
  +m_i \alpha\rho^i \rho^j (u^i_k u^j_\ell + u^i_\ell u^j_k ) + m_j \rho^i(\alpha\rho^j u^j_k u^j_\ell +\frac1\alpha p^j_{k\ell})  \Big\}\\
  + \frac{m_j A_{k\ell}}{(m_i+m_j)^2} \Big(\alpha \rho^i\rho^j \vert\vu^i-\vu^j\vert^2 +\frac3\alpha \rho^j p^i +\frac3\alpha \rho^i p^j \Big)\Bigg) , \qquad \\
  1\leq k,\ell\leq 3, ~1\leq i \leq S.
\end{multline}
\end{itemize}
Moreover, the balance laws \eqref{eq:moment_pressure} imply the following energy balance law for each species
\begin{multline}\label{eq:moment_energy}
  \partial_t (\alpha^3 \rho^i \vert \vu^i\vert^2 +3\alpha p^i) \\
 +\sum_{n=1}^3 \partial_{x_n} (\alpha^3 \rho^i \vert \vu^i\vert^2  u^i_n+ \alpha (2\sum_{k=1}^3 u^i_k p^i_{k n}  + 3u^i_n p^i)) \\
=\sum_{j=1}^S\frac{2\pi \|b^{ij}\|_{L^1}}{(m_i+m_j)^2} \Big\{  \frac1\alpha\left(- 6 m_i\rho^j p^i + 6m_j \rho^i p^j\right)\\\
+\alpha \rho^i\rho^j \left[ -2m_i \vert\vu^i\vert^2 + 2(m_i-m_j) \vu^i\cdot \vu^j +2 m_j  \vert \vu^j\vert^2 \right] \Big\}, \\
  \qquad   ~1\leq i \leq S.
\end{multline}
\end{theorem}

\begin{proof}
We prove each of the balance laws separately, by taking moments of the Boltzmann equations.
 
\noindent $\bullet$ In a very standard way, mass balances are obtained for each species from \eqref{eq:MomentEq-Diffusive} with $\psi^i(\vv) = m_i$,  using the moments of the distribution functions~\eqref{eq:MomentsHigher-DLess0}-\eqref{eq:MomentsHigher-DLess1} and the conservation properties of the collision kernel~\eqref{eq:conservation_collision} to obtain~\eqref{eq:moment_mass}.\smallskip

\noindent $\bullet$  For the momentum balances, the computations are again standard~\cite{BGS}. We choose $\psi^i(\vv) = m_i v_\ell$ in \eqref{eq:MomentEq-Diffusive} for any $1\leq \ell\leq 3$,  and use \eqref{eq:MomentsHigher-DLess1}-\eqref{eq:MomentsHigher-DLess2}, which leads to
\[\alpha \partial_t (\alpha \rho^i u^i_\ell) +\sum_{k=1}^3 \partial_{x_k} (\alpha^2 \rho^i u^i_k u^i_\ell+p^i_{k\ell}) = \frac1\alpha \sum_{j=1}^S \intR m_i v_\ell Q^{ij}(f^i,f^j)(\vv) \dd \vv.\]
Let us compute the term in the right hand side of this last equation. Using the weak form~\eqref{eq:weak_form} with $\phi(\vv) = m_{i} v_\ell$, we obtain 
\begin{multline*}
\mathcal M^{ij}_\ell:= \intR m_i v_\ell Q^{ij}(f^i,f^j)(\vv) \dd \vv\\
  = \frac{m_im_j}{m_i+m_j} \intR\intR f^i(\vv) f^j(\vv_*) ({v_*}_\ell - v_\ell)  \dd\vv_*\dd\vv \intS b^{ij}(\cos\theta) \dd\sigma \\
  + \frac{m_im_j}{m_i+m_j} \intR\intR f^i(\vv) f^j(\vv_*) \vert\vv-\vv_*\vert \dd\vv_*\dd\vv \intS b^{ij}(\cos\theta) \sigma_\ell\dd\sigma , 
\end{multline*}
since 
\[v'_\ell - v_\ell = \frac{m_j}{m_i+m_j} ({v_*}_\ell - v_\ell) + \vert\vv-\vv_*\vert  \sigma_\ell. \]
The two integrals in $\sigma$ are computed using spherical coordinates~\cite{BGS}, and it leads to 
\begin{multline*}
 \mathcal M^{ij}_\ell
  = \frac{2\pi}{m_i+m_j} \|b^{ij}\|_{L^1}\intR\intR m_i f^i(\vv) m_j f^j(\vv_*) ({v_*}_\ell - v_\ell) \dd\vv_*\dd\vv\\
  = \alpha \frac{2\pi}{m_i+m_j} \|b^{ij}\|_{L^1} (\rho^i \rho^j u^j_\ell - \rho^j \rho^i u^i_\ell).
\end{multline*}

The momentum balance laws thus become~\eqref{eq:moment_momentum}.\smallskip

\noindent $\bullet$ To compute the momentum flux balances we choose $\psi^i(\vv) = m_i v_kv_\ell$. From \eqref{eq:MomentEq-Diffusive} for any $1\leq k,\ell \leq 3$, using \eqref{eq:MomentsHigher-DLess2}-\eqref{eq:moments3}, one obtains:
\begin{multline*}
 \alpha \partial_t (\alpha^2 \rho^i u^i_k u^i_\ell+p^i_{k\ell}) +\sum_{n=1}^3 \partial_{x_n} (\alpha^3 \rho^i u^i_k u^i_\ell u^i_n+ \alpha (u^i_k p^i_{\ell n} + u^i_\ell p^i_{nk} + u^i_n p^i{k\ell})) \\
 = \frac1\alpha \sum_{j=1}^S \intR m_i v_k v_\ell Q^{ij}(f^i,f^j)(\vv) \dd \vv. 
\end{multline*}

Again, let us focus on the computation fo the source term, using again the weak form~\eqref{eq:weak_form} with $\phi(\vv) = m_{i} v_kv_\ell$
\begin{multline}\label{eq:Qijell}
 \mathcal Q^{ij}_{k\ell}:= \intR m_i v_kv_\ell Q^{ij}(f^i,f^j)(\vv) \dd \vv\\
 = \intR\intR \intS m_i b^{ij}(\cos \theta) f^i(\vv) f^j(\vv_*) (v'_k v'_\ell - v_kv_\ell) \dd\sigma \dd\vv_*\dd\vv\\
 =\intR\intR \int_0^{2\pi}\int_0^\pi  m_i  f^i(\vv) f^j(\vv_*) (v'_k v'_\ell - v_kv_\ell)b^{ij}(\cos \theta) \sin\theta\dd\theta \dd\varphi   \dd\vv_*  \dd\vv
\end{multline}
where the last equality is the change to spherical coordinates. Using the collision rules, the term $v'_k v'_\ell$ becomes
\begin{multline}\label{eq:vprimek_vprimel}
  v'_k v'_\ell = \frac{1}{(m_i+m_j)^2} (m_i v_k + m_j{v_*}_k + m_j \vert\vv-\vv_*\vert  \sigma_k)  (m_i v_\ell + m_j{v_*}_\ell + m_j \vert\vv-\vv_*\vert  \sigma_\ell)\\
  =\frac{1}{(m_i+m_j)^2} \Big\{ m_i^2 v_kv_\ell + m_im_j (v_k {v_*}_\ell + v_\ell {v_*}_k) +m_j^2{v_*}_k{v_*}_\ell \\
  +  m_j \vert\vv-\vv_*\vert \big( m_i v_k \sigma_\ell + m_j{v_*}_k \sigma_\ell+ m_i v_\ell\sigma_k + m_j {v_*}_\ell\sigma_k \big)\\
  + m_j^2 \vert\vv-\vv_*\vert^2 \sigma_k\sigma_\ell   \Big\}
\end{multline}
Let us first show that by parity arguments, the terms of the last but one line in the previous expression are all zero when integrated as in~\eqref{eq:Qijell}.
Indeed, since all terms involving the velocities do not depend on $\theta$ nor $\varphi$, we have to handle terms of the form
\[\int_0^{2\pi}\int_0^\pi  b^{ij}(\cos \theta) \sin\theta  \sigma_\ell \dd\theta \dd\varphi.\]
These terms can be computed for $\ell=1,2,3$, with $\sigma_1 = \sin\theta \cos\varphi$, $\sigma_2 = \sin\theta\sin\varphi$ and $\sigma_3 = \cos\theta$. The terms for $\ell=1,2$ are obviously zero, because of the periodicity of trigonometric functions in the integration in $\varphi$.
For $\ell=3$, we use the change of variables $\eta = \cos\theta$ to obtain
\[\int_0^{2\pi}\int_0^\pi  b^{ij}(\cos \theta) \sin\theta\cos\theta \dd\theta \dd\varphi = 2\pi \int_{-1}^1 \eta b^{ij}(\eta) \dd\eta = 0,\]
since the function $b^{ij}$ is even.

Now, let us handle the integration of the last term in~\eqref{eq:vprimek_vprimel}.
The only terms depending on $\sigma$ are of the form
\[A_{k\ell} :=\int_0^{2\pi}\int_0^\pi  b^{ij}(\cos \theta) \sin\theta \sigma_k\sigma_\ell \dd\theta \dd\varphi.\]
With the same argument as before, the integration in $\varphi$ leads to zero for the terms with $k\neq \ell$. It remains to handle the terms $A_{kk}$.
We have, using trigonometry relations and the same change of variables as before
\begin{multline*}
 A_{11} = \int_0^\pi  b^{ij}(\cos \theta) \sin^3\theta  \dd\theta \int_0^{2\pi} \cos^2\varphi \dd\varphi \\
 =   \int_0^\pi  b^{ij}(\cos \theta) \sin^3\theta  \dd\theta \int_0^{2\pi} \frac{1+ \cos(2\varphi)}2 \dd\varphi \\
 =\pi \int_{-1}^1 b^{ij}(\eta) (1-\eta^2) \dd \eta = \pi \left(\|b^{ij}\|_{L^1} - B^{ij}\right),
\end{multline*}
where we defined $B^{ij} :=  \int_{-1}^1 \eta^2 b^{ij} (\eta)  \dd \eta $.
With the same reasoning, we can prove that $A_{22} =  \pi \left(\|b^{ij}\|_{L^1} - B^{ij}\right)$.
For $A_{33}$, we have, using again the same change of variables
\[A_{33} =  \int_0^\pi  b^{ij}(\cos \theta) \sin \theta \cos^2\theta \dd\theta \int_0^{2\pi} \dd\varphi =2\pi \int_{-1}^1 \eta^2 b^{ij}(\eta) \dd \eta = 2\pi B^{ij}. 
\]
Therefore, we may summarize:
\begin{equation}
  A_{k\ell} = 
  \begin{cases}
  \pi \left(\|b^{ij}\|_{L^1} - B^{ij}\right), & k = \ell = 1,2; \\
  2\pi B^{ij}, & k = \ell = 3; \\
  0, & k \neq \ell. 
  \end{cases}
\end{equation}

It remains to handle the terms involving only the velocities in~\eqref{eq:vprimek_vprimel}. They all have a common multiplicative factor
\[\int_0^\pi  b^{ij}(\cos \theta) \sin\theta  \dd\theta \int_0^{2\pi} \dd\varphi =2\pi \|b^{ij}\|_{L^1}.\]
Thus, \eqref{eq:Qijell} becomes
\begin{multline*}
 \mathcal Q^{ij}_{k\ell}   =\frac{2\pi \|b^{ij}\|_{L^1}}{(m_i+m_j)^2}\intR\intR  m_i  f^i(\vv) f^j(\vv_*) \\
 \times \Big( m_i^2 v_kv_\ell + m_im_j (v_k {v_*}_\ell + v_\ell {v_*}_k) +m_j^2{v_*}_k{v_*}_\ell 
 - (m_i+m_j)^2 v_kv_\ell\Big)  \dd\vv_*  \dd\vv\\
 +\frac{m_j^2}{(m_i+m_j)^2}\intR\intR  m_i  f^i(\vv) f^j(\vv_*)  (\vert\vv\vert^2 +\vert\vv_*\vert^2-2\vv\cdot\vv_*) A_{kk}\delta_{k\ell} \dd\vv_*  \dd\vv
\end{multline*}
Now, in each term, the variables $\vv$ and $\vv_*$ can be separated. Combined with the moments~\eqref{eq:MomentsHigher-DLess0}-\eqref{eq:MomentsHigher-DLess2} of the distribution functions, it leads to
\begin{multline*}
  \mathcal Q^{ij}_{k\ell}  = \frac{2\pi \|b^{ij}\|_{L^1}}{(m_i+m_j)^2} \Big\{ -(2m_i+m_j)\rho^j(\alpha^2\rho^i u^i_k u^i_\ell + p^i_{k\ell})\\
  +m_i \alpha^2\rho^i \rho^j (u^i_k u^j_\ell + u^i_\ell u^j_k ) + m_j \rho^i(\alpha^2\rho^j u^j_k u^j_\ell + p^j_{k\ell})  \Big\}\\
  + \frac{m_j}{(m_i+m_j)^2} {A_{k\ell}} \Big(\rho^j(\alpha^2 \rho^i\vert\vu^i\vert^2 +\sum_{n=1}^3 p^i_{nn}) + \rho^i(\alpha^2 \rho^j\vert\vu^j\vert^2 +\sum_{n=1}^3 p^j_{nn})\\
  -2 \alpha^2 \rho^i \rho^j \vu^i\cdot\vu^j \Big) 
\end{multline*}
This can also be rewritten as
\begin{multline*}
  \mathcal Q^{ij}_{k\ell}  = \frac{2\pi \|b^{ij}\|_{L^1}}{(m_i+m_j)^2} \Big\{ -(2m_i+m_j)\rho^j(\alpha^2\rho^i u^i_k u^i_\ell + p^i_{k\ell})\\
  +m_i \alpha^2\rho^i \rho^j (u^i_k u^j_\ell + u^i_\ell u^j_k ) + m_j \rho^i(\alpha^2\rho^j u^j_k u^j_\ell + p^j_{k\ell})  \Big\}\\
  + \frac{m_j}{(m_i+m_j)^2}{A_{k\ell}} \Big(\alpha^2 \rho^i\rho^j \vert\vu^i-\vu^j\vert^2 +3\rho^j p^i +3\rho^i p^j \Big) 
\end{multline*}

Finally, this leads to the balance law~\eqref{eq:moment_pressure}.\smallskip

\noindent $\bullet$ Finally, to derive the energy balance laws we shall take into account the assumption that all the species are monatomic gases, and that relations \eqref{eq:Moments-Peculiar}$_{1}$ hold. Thus, energy equations will be derived starting from the balance laws \eqref{eq:moment_pressure}, taking into account $\sum_{k=1}^{3} p_{kk}^{i} = 3 p^{i}$, choosing $\ell=k$ and summing \eqref{eq:moment_pressure} over $k$ to obtain:
\begin{multline}\label{eq:energy}
  \partial_t (\alpha^4 \rho^i \vert \vu^i\vert^2 +3\alpha^2 p^i) \\
 +\sum_{n=1}^3 \partial_{x_n} (\alpha^4 \rho^i \vert \vu^i\vert^2  u^i_n+ \alpha^2 (2\sum_{k=1}^3 u^i_k p^i_{k n}  + 3u^i_n p^i)) \\
 = \sum_{j=1}^S\sum_{k=1}^3 \mathcal Q^{ij}_{kk} = \sum_{j=1}^S\frac{2\pi \|b^{ij}\|_{L^1}}{(m_i+m_j)^2} \Big\{ -(2m_i+m_j)\rho^j(\alpha^2\rho^i \vert \vu^i\vert^2 + 3p^i)\\
  +2m_i \alpha^2\rho^i \rho^j \vu^i\cdot \vu^j  + m_j \rho^i(\alpha^2\rho^j \vert \vu^j\vert^2 + 3p^j) \\
  + m_j\Big(\alpha^2 \rho^i\rho^j \vert\vu^i-\vu^j\vert^2 +3\rho^j p^i +3\rho^i p^j \Big) \Big\} ,
\end{multline}
where we used that $\sum_{k=1}^3 A_{kk} = 2\pi \|b^{ij} \|_{L^1}$.

Let us now compute the term inside the brace in the right-hand side of the previous relation.
\begin{multline*}
 -(2m_i+m_j)\rho^j(\alpha^2\rho^i \vert \vu^i\vert^2 + 3p^i)  +2m_i \alpha^2\rho^i \rho^j \vu^i\cdot \vu^j  + m_j \rho^i(\alpha^2\rho^j \vert \vu^j\vert^2 + 3p^j) \\
  + m_j\Big(\alpha^2 \rho^i\rho^j \vert\vu^i-\vu^j\vert^2 +3\rho^j p^i +3\rho^i p^j \Big)\\
  = \alpha^2 \rho^i\rho^j \left[ -2m_i \vert\vu^i\vert^2 + 2(m_i-m_j) \vu^i\cdot \vu^j +2 m_j  \vert \vu^j\vert^2 \right]\\
  - 6 m_i\rho^j p^i + 6m_j \rho^i p^j.
\end{multline*}
This implies the balance law~\eqref{eq:moment_energy}.
\end{proof}

The energy conservation law for the mixture is a consequence of the energy balances \eqref{eq:moment_energy} and the properties of the collision operator. However, to write it in a form more similar to the usual macroscopic equations, one needs to define appropriate mixture variables \cite{truesdell1984rational,ruggeri2007hyperbolic}. To that end, we define the mass density $\rho$, momentum density $\rho u_{k}$, internal energy density $\rho \varepsilon$, pressure tensor $p_{kn}$ and internal energy flux $q_{n}$, $1\leq k,n\leq 3$ of the mixture in dimensionless form: 
\begin{align} \label{eq:Mixture-DensitiesFluxes}
  \rho & := \sum_{i=1}^{S} \rho^{i}, \quad 
	\rho u_{k} := \sum_{i=1}^{S} \rho^{i} u^{i}_{k}, 
  \nonumber \\ 
  \rho \varepsilon & := \sum_{i=1}^{S} \rho^{i} \varepsilon^{i} 
	+ \alpha^{2} \sum_{i=1}^{S} \frac{1}{2} \rho^{i} \vert \vu^{i} - \vu \vert^{2}, 
  \nonumber \\ 
  p_{kn} & := \sum_{i=1}^{S} p_{kn}^{i} 
    + \alpha^{2} \sum_{i=1}^{S} \rho^{i} (u_{k}^{i} - u_k) (u_{n}^{i}-u_n), 
  \nonumber \\ 
  q_{n} & := \sum_{i=1}^{S} \left( \rho^{i} \varepsilon^{i} 
    + \alpha^{2} \frac{1}{2} \rho^{i} \vert \vu^{i} - \vu\vert^2 \right) (u_{n}^{i}-u_n)
    + \sum_{i=1}^{S} \sum_{k=1}^{3} p_{kn}^{i} (u_{k}^{i}-u_k) .
\end{align}
We can the state the energy conservation law for the mixture.

\begin{proposition} \label{th:Energy-conservation}
 The energy balance law \eqref{eq:moment_energy} also implies a conservation law for the whole mixture
 \begin{multline} \label{eq:claw_energy}
  \alpha^{3} \partial_t \left( \sum_{i=1}^{S} \frac{1}{2} \rho^i \vert \vu^i\vert^2 \right) + \alpha \partial_t \left( \frac{3}{2} \sum_{i=1}^{S} p^i \right)   + \alpha^{3} \sum_{n=1}^3 \partial_{x_n} \left( \sum_{i=1}^{S} \frac{1}{2} \rho^i \vert \vu^i\vert^2  u^i_n \right) \\
+ \alpha \sum_{n=1}^3 \partial_{x_n} \left( 	\sum_{i=1}^{S} \sum_{k=1}^3 u^i_k p^i_{k n}  + \frac{3}{2} \sum_{i=1}^{S} u^i_n p^i \right) 	= 0.
\end{multline}
It can be rewritten under the following form
\begin{multline} \label{eq:claw_energy-traditional}
  \partial_t \left( \alpha^{3} \rho \vert \vu \vert^2 + 2 \alpha \rho \varepsilon \right)  \\
  + \sum_{n=1}^3 \partial_{x_n} \left\{ \left[ \alpha^{3} \rho \vert \vu \vert^{2} 
    + 2 \alpha \rho \varepsilon \right] u_{n} 
    + 2 \alpha \sum_{k=1}^{3} p_{kn} u_{k} + 2 \alpha q_{n} \right\} = 0.
\end{multline}
\end{proposition}

\begin{proof}
From the computation of the source term $\mathcal Q^{ij}_{k\ell}$, we can  deduce the following relations using symmetry with respect to $i$ and $j$
\[\sum_{k=1}^3 \mathcal Q^{ii}_{kk} = 0, \qquad\qquad \sum_{k=1}^3 \mathcal Q^{ij}_{kk}  + \sum_{k=1}^3 \mathcal Q^{ji}_{kk}  =0.\]
Therefore, summation of \eqref{eq:moment_energy} over $i$ annihilates the right-hand side and one ends up with the energy conservation law for the mixture \eqref{eq:claw_energy}.
To recast the conservation law \eqref{eq:claw_energy-traditional}, we use appropriate macroscopic mixture variables \eqref{eq:Mixture-DensitiesFluxes}. Having in mind the relation 
 $ 3 p^{i} = 2 \rho^{i} \varepsilon^{i}$, 
the  energy density of the mixture may be written as:
\begin{equation}
  \alpha^3 \sum_{i=1}^{S} \rho^i \vert \vu^i\vert^2 + 3 \alpha \sum_{i=1}^{S} p^i 
    = \alpha^{3} \rho \vert \vu \vert^2 + 2 \alpha \rho \varepsilon.
\end{equation}
On the other hand, the energy flux of the mixture may be transformed to
\begin{multline*}
  \alpha^3 \sum_{i=1}^{S} \rho^i \vert \vu^i\vert^2  u^i_n 
    + \alpha \left( 2 \sum_{i=1}^{S} \sum_{k=1}^3 u^i_k p^i_{k n}  
    + 3 \sum_{i=1}^{S} u^i_n p^i \right) \\
 = \left[ \alpha^{3} \rho \vert \vu \vert^{2} + 2 \alpha \rho \varepsilon \right] u_{n} 
    + 2 \alpha \sum_{k=1}^{3} p_{kn} u_{k} + 2 \alpha q_{n}.
\end{multline*}
In such a way energy conservation law \eqref{eq:claw_energy} becomes~\eqref{eq:claw_energy-traditional}.
\end{proof}

\subsection{Asymptotic limit} 

We formally consider the asymptotic limit $\alpha \to 0$ of the scaled moment equations. In the case of local equilibrium approximation, one obtains the classical Maxwell--Stefan equations or their non-isothermal counterpart. Our aim is to formally analyze the asymptotic limit for the higher-order system of moment equations. First, an isothermal model will be analyzed to underline the difference between the standard Maxwell--Stefan model and the higher-order one. After that, an appropriate non-isothermal extension without heat conduction will be discussed. 

\begin{theorem} \label{th:Higher-order_M-S}
Formally, when the scaling parameter $\alpha$ tends to zero, the macroscopic quantities for each species $\rho^i$, $u^{i}_{k}$ and $p^i_{k\ell} = p^i \delta_{k\ell} + p^i_{\langle k\ell\rangle}$ defined by \eqref{eq:MomentsHigher-DLess0}-\eqref{eq:MomentsHigher-DLess2} satisfy the following system
\begin{equation}\label{eq:systeme_alpha_0}
 \begin{cases}
 \partial_t \rho^i + \sum_{k=1}^{3} \partial_{x_k} (\rho^i u^{i}_{k}) =0,\\
  \partial_{x_\ell} \left( p^i_{\ell\ell}\right) = \sum_{j=1}^S \frac{2\pi\|b^{ij}\|_{L^1}}{m_i+m_j} \rho^i\rho^j(u^j_\ell - u^i_\ell), ~1\leq\ell\leq3\\
  p^i_{\langle k\ell\rangle} =0, ~1\leq k\neq \ell\leq 3\\
  p^i_{\langle \ell\ell\rangle} = [M^{-1}\beta^{\ell\ell}]_i , ~1\leq  \ell\leq 3,
\end{cases}
\end{equation}
where $[M^{-1}\beta^{\ell\ell}]_i$ is the $i$-th component of the product of the $S\times S$ matrix $M$ defined by 
\[
M_{ij} =
\begin{cases}
 \dfrac{2\pi\|b^{ij}\|_{L^1}}{(m_i+m_j)^2} m_j \rho^i, & \text{if $j\neq i$},\\
 \dfrac{2\pi\|b^{ii}\|_{L^1}}{4m_i^2} m_i \rho^i -\displaystyle\sum_{j=1}^S \dfrac{2\pi\|b^{ij}\|_{L^1}}{(m_i+m_j)^2} (2m_i+m_j) \rho^j , & \text{if $j= i$},
\end{cases}
\]
and the vector $\beta^{\ell\ell}$ defined by
\begin{multline*}
  \beta^{\ell\ell}_i = \sum_{j=1}^S \frac{\pi }{(m_i+m_j)^2} \times\\
\begin{cases}
\Big[\|b^{ij}\|_{L^1}  \Big( (m_j-4m_i)\rho^j p^i  + 5m_j \rho^i p^j \Big)  -3 m_j B^{ij} \Big(\rho^j p^i +\rho^i p^j \Big) \Big]   ,  & \text{if $\ell=1,2$}\\
2\Big[  \|b^{ij}\|_{L^1}\Big( -(2m_i+m_j)\rho^j p^i  + m_j \rho^i p^j \Big) +3 m_j B^{ij} \Big(\rho^j p^i +\rho^i p^j \Big) \Big]  , & \text{if $\ell=3$}.
\end{cases}
\end{multline*}
In this system, the partial pressures $p^i$ for each species have to be given as functions of $\rho^i$ by an equation of state.
\end{theorem}

\begin{proof}
The first relation of  \eqref{eq:systeme_alpha_0} is obvious, and comes straightforwardly from \eqref{eq:moment_mass}.
Let us start with the proof of the last two relations of \eqref{eq:systeme_alpha_0}.
At order $\alpha^0$, the momentum flux balance law \eqref{eq:moment_pressure} gives
\begin{multline*}
 \sum_{j=1}^S\frac{2\pi \|b^{ij}\|_{L^1}}{(m_i+m_j)^2} \Big\{ -(2m_i+m_j)\rho^j p^i_{k\ell}
   + m_j \rho^i p^j_{k\ell} \Big\}\\
  + \frac{3 m_j A_{kk}\delta_{k\ell}}{(m_i+m_j)^2} \Big(\rho^j p^i +\rho^i p^j \Big) =0
\end{multline*}
Using the decomposition of the pressure into its diagonal and traceless part, we obtain
\begin{multline}\label{eq:pkl_specieswise}
 \sum_{j=1}^S\frac{2\pi \|b^{ij}\|_{L^1}}{(m_i+m_j)^2} \Big\{ -(2m_i+m_j)\rho^j p^i_{\langle k\ell\rangle}
   + m_j \rho^i p^j_{\langle k\ell\rangle} \Big\}\\
  + \delta_{k\ell}\Big\{ \frac{2\pi \|b^{ij}\|_{L^1}}{(m_i+m_j)^2} \Big( -(2m_i+m_j)\rho^j p^i  + m_j \rho^i p^j \Big)
  + \frac{3 m_j A_{kk}}{(m_i+m_j)^2} \Big(\rho^j p^i +\rho^i p^j \Big) \Big\}  =0
\end{multline}
For fixed $k,\ell$, observe that these relations, for any $1\leq i\leq S$, constitute a linear system for the unknowns $p^i_{\langle k\ell\rangle}$. 
Indeed, if we introduce the vector of pressures for all species $\mathbb P_{\langle k\ell\rangle} = (p^1_{\langle k\ell\rangle},p^2_{\langle k\ell\rangle},\cdots,p^S_{\langle k\ell\rangle})^T$, equations \eqref{eq:pkl_specieswise} can be rewritten as 
\[ M \mathbb P_{\langle k\ell\rangle} =\delta_{k\ell} \beta^{\ell\ell}.\]

Let us prove that the matrix $M$ is invertible, by proving that its transpose is diagonally dominant, \textit{i.e.} that for any $1\leq i \leq S$, $\vert M_{ii}\vert >\sum_{j\neq i} \vert M_{ji}\vert$.
To this end, let us split the sum over $j$ into the sum for $j\neq i$ and the term for $j=i$ in $M_{ii}$ to rewrite
\[ M_{ii} = -\Big( \frac{\pi\|b^{ii}\|_{L^1}}{m_i} \rho^i +\sum_{j\neq i} \frac{2\pi\|b^{ij}\|_{L^1}}{(m_i+m_j)^2} (2m_i+m_j) \rho^j \Big).\]
Since all coefficients in the matrix $M$ are non-negative, we can compute
\[
 \vert M_{ii}\vert -\sum_{j\neq i} \vert M_{ji}\vert = \frac{\pi\|b^{ii}\|_{L^1}}{m_i} \rho^i +\sum_{j\neq i} \frac{2\pi\|b^{ij}\|_{L^1}}{(m_i+m_j)^2} (m_i+m_j) \rho^j ,
\] 
which is positive as soon as one $\rho^i$ is positive. This proves the invertibility of the matrix $M$.
We thus obtain that for any $1\leq i\leq S$ and $1\leq k\neq \ell\leq 3$, $p^i_{\langle k\ell\rangle} = 0$.
Further, the values $p^i_{\langle 11\rangle}=p^i_{\langle 22\rangle} $ and $p^i_{\langle 33\rangle}$ are non-identically zero for all species $1\leq i \leq S$, and can be obtained through the inverse of $M$ and $\beta^{\ell\ell}$ as
\begin{equation}\label{eq:Pkk}
\mathbb P_{\langle \ell\ell\rangle} = M^{-1}\beta^{\ell\ell}, \qquad 1\leq \ell \leq 3.
\end{equation}
Thus, the pressure tensor is diagonal, and 
\begin{equation}\label{eq:pikl}
p^i_{k\ell} = (p^i + p^i_{\langle kk\rangle}) \delta_{k\ell}, \qquad 1\leq i \leq S, ~1\leq k,\ell \leq 3.
\end{equation}

We can now obtain the second relation in \eqref{eq:systeme_alpha_0}. The formal limit of the momentum balance law \eqref{eq:moment_momentum} at order $\alpha^0$ gives
\begin{equation*}
\sum_{k=1}^3 \partial_{x_k} p^i_{k\ell} = \sum_{j=1}^S \frac{2\pi\|b^{ij}\|_{L^1}}{m_i+m_j} \rho^i\rho^j(u^j_\ell - u^i_\ell) 
\end{equation*}
Using \eqref{eq:Pkk} and \eqref{eq:pikl} in the previous equation, we obtain
\begin{equation}\label{eq:limitMS}
 \partial_{x_\ell} \left( p^i+p^i_{\langle\ell\ell\rangle}\right) = \sum_{j=1}^S \frac{2\pi\|b^{ij}\|_{L^1}}{m_i+m_j} \rho^i\rho^j(u^j_\ell - u^i_\ell), \qquad 1\leq \ell \leq 3.
\end{equation}
Moreover, observe that if we sum over $i$, the right-hand side cancels, and it leads to
\[ \partial_{x_\ell} \left( p^i+\sum_{i=1}^S p^i_{\langle\ell\ell\rangle}\right) =0, \qquad 1\leq \ell \leq 3. \]
\end{proof}

Theorem \ref{th:Higher-order_M-S} inherits a compatibility condition which restricts the structure of partial pressures $p^{i}$, which is summarized in the following Proposition.

\begin{proposition} \label{th:Ideal-gas}
In order for the system \eqref{eq:systeme_alpha_0} to have a solution, the equation of state giving the pressures $p^i$ has to satisfy the following compatibility condition
\begin{equation}\label{eq:compat}
m_i \frac{p^i}{\rho^i} = \theta(t,x), 
\qquad 1\leq i \leq S.
\end{equation}
This condition is in particular satisfied by the scaled ideal gas law: 
\begin{equation}\label{eq:EOSideal}
  p^i(t,\vx) = \frac53 \rho^i(t,\vx) \frac{T(t,\vx)}{m_i},
\end{equation}
where $T$ is the common temperature of the species. 
\end{proposition}
\begin{proof}
 In equations \eqref{eq:pkl_specieswise}, choosing $k=\ell$ and summing over $k$ cancels the terms involving $p^i_{\langle kk\rangle}$, and for any $1\leq i \leq S$, it leads to
\begin{multline*}
  \sum_{j=1}^S \frac{6\pi \|b^{ij}\|_{L^1}}{(m_i+m_j)^2} \Big\{ -(2m_i+m_j)\rho^j p^i  + m_j \rho^i p^j 
  +  m_j (\rho^j p^i +\rho^i p^j ) \Big\}  \\
  =  \sum_{j=1}^S \frac{12\pi \|b^{ij}\|_{L^1}}{(m_i+m_j)^2} \Big(-m_i \rho^j p^i  + m_j \rho^i p^j  \Big) =0,
\end{multline*}
since $\sum_{k=1}^3 A_{kk} = 2\pi\|b^{ij}\|_{L^1}$.
This means that 
\begin{equation}\label{eq:relationspi}
  \sum_{j=1}^S \frac{\|b^{ij}\|_{L^1}}{(m_i+m_j)^2} \Big(m_j \rho^i p^j  - m_i \rho^j p^i   \Big) =0, 
  \quad i = 1,\ldots,S. 
\end{equation}
These $S$ equations are linearly dependent, since the sum of all these relations is zero. 
Moreover, we can prove that they imply the compatibility condition \eqref{eq:compat}.
Indeed, let $\zeta^i:=m_i \frac{p^i}{\rho^i} - m_S \frac{p^S}{\rho^S}$. Then, we can replace $p^i$ in \eqref{eq:relationspi} for $1\leq i \leq S-1$, which gives
\[ \rho^i \sum_{j=1}^S  \frac{\|b^{ij}\|_{L^1}}{(m_i+m_j)^2} \rho^j \left(\zeta^j+m_S \frac{p^S}{\rho^S}\right) - \rho^i \left(\zeta^i + m_S \frac{p^S}{\rho^S}\right)\sum_{j=1}^S  \frac{\|b^{ij}\|_{L^1}}{(m_i+m_j)^2}\rho^j =0,\]
and for $i=S$
\[ \rho^S \sum_{j=1}^S  \frac{\|b^{ij}\|_{L^1}}{(m_i+m_j)^2} \rho^j \left(\zeta^j+m_S \frac{p^S}{\rho^S}\right) - \rho^S m_S \frac{p^S}{\rho^S}\sum_{j=1}^S  \frac{\|b^{ij}\|_{L^1}}{(m_i+m_j)^2}\rho^j =0.\]
This last equation implies that 
\[  \frac{\|b^{ij}\|_{L^1}}{(m_i+m_j)^2} \rho^j \zeta^j =0,\]
which in turns implies in the previous relations, for $1\leq i \leq S-1$
\[ \rho^i \sum_{j=1}^S  \frac{\|b^{ij}\|_{L^1}}{(m_i+m_j)^2} \rho^j m_S \frac{p^S}{\rho^S} - \rho^i \zeta^i \sum_{j=1}^S  \frac{\|b^{ij}\|_{L^1}}{(m_i+m_j)^2}\rho^j  + \rho^i m_S \frac{p^S}{\rho^S}\sum_{j=1}^S  \frac{\|b^{ij}\|_{L^1}}{(m_i+m_j)^2}\rho^j=0,\]
which means that for any $1\leq i \leq S-1$
\[\zeta^i \sum_{j=1}^S  \frac{\|b^{ij}\|_{L^1}}{(m_i+m_j)^2}\rho^j  =0.\]
This implies that $\zeta^i = 0$ for any $1\leq i \leq S-1$, which completes the proof.
\end{proof} 

The main outcome of Theorem \ref{th:Higher-order_M-S} (and Proposition \ref{th:Ideal-gas}), describing the higher-order Maxwell--Stefan model in the asymptotic limit, is that momentum  exchange between the species is not balanced solely by the gradients of partial pressures. One must take into account the diagonal terms of the traceless part of the pressure tensor, determined by the set of algebraic relations. 

In Theorem \ref{th:Higher-order_M-S}, the system \eqref{eq:systeme_alpha_0} is a closed one when the partial pressures $p^i$ are solely determined by the densities, meaning that the quantity $\theta$ in \eqref{eq:compat} is a known constant. This corresponds to the isothermal case. In the non-isothermal case, the temperature varies and is an additional unknown of the system. In order to close the system, we need an additional equation, which is given by the asymptotic limit of the energy conservation for the mixture \eqref{eq:claw_energy}
\[ \partial_t \left( \frac{3}{2} \sum_{i=1}^{S} p^i \right)   +  \sum_{n=1}^3 \partial_{x_n} \left( 	\sum_{i=1}^{S} \sum_{k=1}^3 u^i_k p^i_{k n}  + \frac{3}{2} \sum_{i=1}^{S} u^i_n p^i \right) 	= 0.\]

\subsection{Higher-order diffusion model} 

In the final step of higher-order description of diffusion, a more general model than the one obtained in the asymptotic limit is proposed. To motivate the forthcoming analysis, let us mention that in the mixture models relevant for diffusion, governing equations consist of the balance laws of masses and the balance laws of momenta for the species. They can be derived either in a systematic way \cite{muller1998rational,bothe2015continuum,ruggeri2021classical}, or in an \emph{ad hoc} manner \cite{kerkhof2005analysis}. In the latter case, even viscous dissipation is introduced by assumption. 

Our analysis will be based upon scaled moment equations given in Theorem \ref{th:Moment-Equations} and Proposition \ref{th:Energy-conservation}, i.e. mass balances \eqref{eq:moment_mass}, momentum balances \eqref{eq:moment_momentum}, momentum flux balances \eqref{eq:moment_pressure} and energy conservation \eqref{eq:claw_energy}. They provide a clear insight into the order of magnitude of all terms needed for the construction of generalized diffusion models. In particular, to build up the model which inherits the inertia terms in the momentum balance equations \eqref{eq:moment_momentum} one has to keep all the $O(\alpha^{2})$ terms. This corresponds to the models of diffusion proposed in \cite{bothe2015continuum,kerkhof2005analysis}. If we want to expand this to higher order models, we have to determine the order of magnitude of terms which will be kept in the governing equations. The model we propose here will be limited to the $O(\alpha^{2})$ terms in all equations. 

\begin{proposition} \label{th:Higher-order_Diffusion}
The higher-order diffusion model which inherits up to $O(\alpha^{2})$ terms consists of the following set of equations: 
\begin{equation}\label{eq:Mass-proper-Final} 
 \alpha \partial_t\rho^i + \alpha \sum_{k=1}^3 \partial_{x_k} ( \rho^i u^i_k ) = 0,
\end{equation}
\begin{equation}\label{eq:Momentum-proper-Final}
  \alpha^2\left[ \partial_t (\rho^i u^i_\ell) +\sum_{k=1}^3 \partial_{x_k} ( \rho^i u^i_k u^i_\ell)  \right] +\sum_{k=1}^3 \partial_{x_k} p^i_{k\ell} =\sum_{j=1}^S   \frac{2\pi \|b^{ij}\|_{L^1}}{m_i+m_j} \rho^i \rho^j (u^j_\ell - u^i_\ell), 
\end{equation}
\begin{multline}\label{eq:MomentumFlux-proper-Final}
  \alpha \, \partial_t p^i_{k\ell} 
+ \alpha \sum_{n=1}^3 \partial_{x_n} (u^i_k p^i_{\ell n} + u^i_\ell p^i_{nk} + u^i_n p^i_{k\ell}) \\
 = \alpha \sum_{j=1}^S\frac{2\pi \|b^{ij}\|_{L^1}}{(m_i+m_j)^2} \rho^{i} \rho^{j} \Big\{ -(2m_i+m_j) u^i_k u^i_\ell
 + m_i ( u^i_k u^j_\ell + u^i_\ell u^j_k ) + m_j u^j_k u^j_\ell \Big\} \\
 + \alpha \sum_{j=1}^S\frac{m_j A_{k\ell}}{(m_i+m_j)^2} \rho^i \rho^j \vert\vu^i - \vu^j \vert^2\\
 + \frac{1}{\alpha} \sum_{j=1}^S\frac{2\pi \|b^{ij}\|_{L^1}}{(m_i+m_j)^2} 
\Big\{ -(2m_i+m_j)\rho^j p^i_{k\ell} + m_j \rho^i p^j_{k\ell}  \Big\} \\
 + \frac{1}{\alpha} \sum_{j=1}^S\frac{m_j A_{k\ell}}{(m_i+m_j)^2} \Big( 3 \rho^j p^i + 3 \rho^i p^j \Big),
\end{multline}
\begin{equation}\label{eq:Energy-proper-Final}
 \alpha \partial_t \left( \frac{3}{2} \sum_{i=1}^{S} p^i \right) + \alpha \sum_{n=1}^3 \partial_{x_n} \left( 
\sum_{i=1}^{S} \sum_{k=1}^3 u^i_k p^i_{k n}  + \frac{3}{2} \sum_{i=1}^{S} u^i_n p^i \right) = 0.
\end{equation}
\end{proposition} 

\begin{proof}
The mass balances \eqref{eq:Mass-proper-Final} and the momentum balances \eqref{eq:Momentum-proper-Final} are the exact moment equations \eqref{eq:moment_mass} and \eqref{eq:moment_momentum}, respectively. The momentum flux balances \eqref{eq:MomentumFlux-proper-Final} and the energy conservation \eqref{eq:Energy-proper-Final} are obtained from the moment equations \eqref{eq:moment_pressure} and \eqref{eq:claw_energy}, respectively, by neglecting the terms $O(\alpha^{3})$. 
\end{proof}

It must be emphasized that the model \eqref{eq:Mass-proper-Final}-\eqref{eq:Energy-proper-Final} is independent of the higher-order Maxwell--Stefan model obtained in the asymptotic limit. In particular, $O(\alpha^{-1})$ terms in equation \eqref{eq:MomentumFlux-proper-Final} (or in equation \eqref{eq:moment_pressure}) vanish in the asymptotic limit, and determine the higher-order corrections to the Maxwell--Stefan model as a solution of the system of algebraic equations (see Theorem \ref{th:Higher-order_M-S}). However, one cannot discard the same terms in equation \eqref{eq:MomentumFlux-proper-Final}, since these are now the genuine balance laws, i.e. the rate type equations, which determine the behavior of pressure tensors $p^{i}_{k\ell}$. 

\begin{remark}
The generalized Maxwell--Stefan model \eqref{eq:Mass-proper-Final}-\eqref{eq:Energy-proper-Final} is non-isothermal by assumption. In the non-isothermal models of previous works \cite{hutridurga2017maxwell,anwasia2022maximum}, all the species bear the same temperature $T$.  Nevertheless, it is a natural question whether the multi-temperature assumption is more appropriate in a generalized setting. 
Because of the terms of order $O(\alpha)$ in \eqref{eq:MomentumFlux-proper-Final}, the same reasoning as in Proposition \ref{th:Ideal-gas} would lead to \eqref{eq:compat} with a correction of order $O(\alpha^{2})$. However, since $p^{i} \sim T^{i}$, it turns out that $O(\alpha^{2})$ terms in $T^{i}$ will eventually bring $O(\alpha^{3})$ contributions to the energy conservation \eqref{eq:Energy-proper-Final}, and  can thus be neglected in this approximation. As a conclusion, single-temperature assumption is still appropriate in the generalized Maxwell--Stefan model when limited to $O(\alpha^{2})$ terms. 
\end{remark}

\section{Conclusion} 
\label{sec:Conclusion}

In this study, the higher-order Maxwell--Stefan diffusion model is derived starting from the kinetic theory of mixtures. It was based upon moment equations for mass, momentum and momentum flux of the species in isothermal case, adjoined with energy conservation law for the mixture in non-isothermal case. All the equations were analyzed in the diffusive scaling. They were closed by the use of an approximate form of the velocity distribution function, obtained by means of maximum entropy principle. The aim was to incorporate the viscous dissipation in the model, in a consistent way, through a proper asymptotic analysis. 

The higher-order model was derived in two different forms. The first form was obtained in the asymptotic limit. It generalizes the classical Maxwell--Stefan model by including the diagonal terms of deviatoric part of the partial stress tensors, which turn out to be linear functions of partial pressures (Theorem \ref{th:Higher-order_M-S}). In such a way, the coupling in Maxwell--Stefan relations does not occur only through diffusion fluxes, but also through the gradients of partial pressures. 

The second form was obtained when terms of order $O(\alpha^{2})$ were retained in moment equations, and higher-order terms were neglected (Theorem \ref{th:Higher-order_Diffusion}). This results in a set of equations containing complete balance laws of mass and momentum of species, and truncated balance laws of momentum fluxes. On one hand, such a model includes complete partial pressure tensors and determines their behavior through the rate-type equations. On the other hand, it presents a consistent generalization of the diffusion model with respect to the order of magnitude of terms which appear in governing equations. 

The higher-order model presented in this study opens several possible paths for further analysis. First, it relies on the moment method applied to the set of Boltzmann equations, and the viscous dissipation is included through the balance laws. It would be of interest to recover the approximation of partial stress tensors for Newtonian fluids. Second, the generalization obtained in the asymptotic limit leads to stronger coupling of equations. It is interesting to see the implications of this coupling through some numerical simulations and to compare the obtained results with a classical Maxwell--Stefan model.

\bibliography{M-S_bibliography}

\end{document}